\documentclass[11pt,a4paper]{amsart} 
\usepackage{amsmath,amsthm,amssymb,amsfonts,color}
\usepackage{framed}
\usepackage{verbatim}
\usepackage{hyperref}

\newtheorem{theorem}{Theorem}[section]
\newtheorem{corollary}[theorem]{Corollary}
\newtheorem{lemma}[theorem]{Lemma}

\newtheorem{proposition}[theorem]{Proposition}
\theoremstyle{definition}
\newtheorem{remark}[theorem]{Remark}
\newtheorem{definition}[theorem]{Definition}

\def\ds{\displaystyle}
\def\R{\mathbb{R}}
\def\N{\mathbb{N}}

\def\dist{{\textup{dist}}}
\def\H{\mathcal{H}}

\def\Om{\Omega}
\def\Phio{{\Phi^\circ}}
\def\ph{\varphi}
\def\e{\varepsilon}
\def\Anis{\mathcal{N}}
\def\vect{\textup{vect}}
\def\kma{\overline{k}}
\def\mma{\overline{m}}
\def\As{\mathcal{A}}

\author{Antonin Chambolle}
\address{CMAP, Ecole Polytechnique, CNRS, 91128 Palaiseau, France}
\email{antonin.chambolle@cmap.polytechnique.fr}

\author{Matteo Novaga}
\address{Department of Mathematics, University of Pisa, Pisa, Italy} 
\email{novaga@dm.unipi.it} 

\author{Berardo Ruffini}
\address{Institut Montpelli\'erain Alexander Grothendieck, University of Montpellier, CNRS, 34095 Montpellier Cedex 5, France}
\email{berardo.ruffini@umontpellier.fr}

\title{Some results on anisotropic fractional mean curvature flows}

\begin{document}

\begin{abstract}
We show the consistency of a threshold dynamics type algorithm for the anisotropic motion by fractional mean curvature, in the presence of a time dependent forcing term.
Beside the consistency result, we  show that convex sets remain convex during the evolution, 
and the evolution of a bounded convex set is uniquely defined.
\end{abstract}

\maketitle 
\tableofcontents

\section{Introduction}\label{secintro}

In this paper we study the evolution of a hypersurface by 
anisotropic fractional mean curvature with the addition of a time-dependent forcing term. 
Such nonlocal evolutions have been first considered in \cite{barlesimbert,imbert}, 
where existence and comparison of weak solutions is proved, 
by suitably adapting the viscosity theory to (geometric) nonlocal equations. 
These results have been later extended in \cite{chamorpon} to more general
(yet translation--invariant) equations.
We point out that an existence and uniqueness result for smooth solutions is still 
not available, even if some results in this direction can be found in \cite{SV}.

In \cite{cafsou}, the authors prove the convergence to the (isotropic) 
motion by fractional mean curvature of a threshold dynamics scheme,
analogous to the one introduced in  \cite{mbo} in the local case.

In this paper we extend their result to the anisotropic case and to the presence of an external driving force.
More precisely, we consider a slightly modified scheme defined by 
an anisotropic convolution followed by a thresholding,
in the spirit of what was proposed in \cite{ishiipiressouganidis}
in the local case,
and we show the convergence of the scheme to a viscosity solution of a
geometric equation, at least when this solution is unique (which is
generally the case~\cite{imbert}). The limit equation is a flow by
anisotropic fractional curvature with a forcing term. Such curvature
corresponds to the first variation
of an anisotropic fractional perimeter of the form introduced in \cite{ludwig}.


We then prove that our scheme is convexity preserving, so that as a consequence
also the limit geometric evolution preserves convexity.
This is a well-known property of the (anisotropic) mean curvature flow 
(see \cite{huisken,ggis,andrews,belcaschanov1}), 
but was not previously known in the fractional case,
both isotropic and anisotropic. 
Eventually, we deduce that convex evolutions are necessarily unique.

The plan of the paper is as follows:
in Section~\ref{secscheme} we introduce the geometric flow, the discrete approximation scheme and we recall some definitions, in particular that of viscosity solution (previously introduced in \cite{barlesimbert,imbert,chamorpon}).   
In Section~\ref{secconv} we establish
the convergence of the scheme to a viscosity solution. This is done in Theorem~\ref{mainthm} and Proposition~\ref{consistency}.
In Section~\ref{secconvex}, building upon known results on convex bodies~\cite{gardner}, we show that the discrete scheme preserves the convexity of a set, and, as a consequence, also the level set equation results convexity preserving.
In Section~\ref{secsplit}, and in particular in Proposition \ref{splittingthm}, we show that the limit motion
can be obtained by alternating curvature motions without forcing term, and evolutions
with the forcing term only. This technical result allows to estimate easily
the relative evolution of two sets with different forcing terms.
Thanks to this estimate, we can deduce in Section~\ref{secunique} the geometric uniqueness of convex evolutions. 
Eventually, in Section~\ref{secremarks}, we state some final considerations and open problems.

\section{Preliminaries and the time-discrete scheme}\label{secscheme}
\subsection{The scheme and the limit equation}
Let $\Anis :\R^N\to \R$ be a norm (that is a convex, even, one-homogeneous
function), with in particular
\begin{equation}\label{aniscontrol}
\underline{c}|x|\le\Anis (x)\le \overline{c}|x|
\end{equation}
for every $x\in \R^N$, where $\overline c$ and $\underline c$ are suitable positive constants.
Given $s\in (0,1)$ and $h>0$, we let throughout the paper
$\sigma_h=h^\frac{s}{1+s}$ and define the kernels
\begin{equation}\label{nuclei}
P(x):=\frac{1 }{1+\Anis (x)^{N+s}}\quad\quad \text{and} \quad\quad  
P_h(x):=\frac{1}{\sigma_h^{N/s}}\,P\left(\frac{x}{\sigma_h^{1/s}}\right)\quad \text{for }h>0\,,
\end{equation}
so that $\sigma_h^{-1}P_h$ converges to $\Anis^{-(N+s)}$ in $L^1_{{\rm loc}}(\R^N\setminus\{0\})$, as $h\to 0$.

For a measurable set $E\subset \R^N$ and $g\in C^0(\R_+)$, where $\R_+:=[0,+\infty)$,
we consider the scheme
\begin{equation}\label{schema}
T_{g(nh),h}(E) := \left\{P_{h}*(\chi_E-\chi_{E^c}) > g(nh) h^{\frac{s}{1+s}}\right\}.
\end{equation}
Given a closed set $E_0$,
we wish to study the limit, as $h\to 0$, of the iterates 
\begin{equation}\label{eq:iterates}
T_{g(nh),h}T_{g((n-1)h),h}\dots T_{g(h),h}(E_0).
\end{equation}
This scheme is a nonlocal variant of the celebrated
Merriman-Bence-Osher scheme~\cite{mbo}, in a form
which has been studied  in~\cite{cafsou} in the context of fractional curvature
flows, and  in~\cite{ishiipiressouganidis}
(see also~\cite{chanov,lauxswartz}) in the context of convolution-generated
motions with a forcing term.
The limit, as we will see, satisfies a nonlocal anisotropic
mean curvature flow with forcing term, which we will introduce below.

Adopting  the notation
of~\cite{cafsou}, as we will do in the whole paper, we define inductively a function
$u_h:\R^N\times\{nh\}_{n\in\mathbb N}\to\R$ as follows:
\[
u_h(\cdot,0)=\widetilde\chi_{E_0}:=\chi_{E_0}-\chi_{\R^N\setminus {E_0}},
\]
\[
u_h(\cdot,(n+1)h)=\widetilde \chi_{\big\{P_h*u_h(\cdot,nh)\ge g(nh) h^\frac{s}{1+s}\big\}}.
\]
The function $u_h$ is then extended to $\R^N\times\R_+$ by letting
$u_h(x,t)=u_h(x,[t/h]h)$ for $t\ge 0$, where $[\cdot]$ denotes
the integer part. 

When $\Anis=|\cdot|$ and $g=0$, it is proved in~\cite{cafsou} that
as $h\to 0$, $u_h$ converges to the geometric solution of the
\textit{fractional curvature flow} defined in~\cite{imbert}, at
least when no ``fattening'' occurs. We shall extend
this result to a more general setting, that is,
with arbitrary norm $\Anis$ and a time varying (continuous) forcing
term $g$. The equation which is solved in the limit is a ``level-set''
equation (an equation which describes the geometric motion of
the level sets of a function) which  must be understood in the viscosity
sense, the precise definition will be given in Section~\ref{sec:visco}
below. In our setting the limit solves the following level-set equation:
\begin{equation}\label{nlMCF}
\partial_t u = \As  (Du)|Du|\left( -\kappa_s(x,\{u\ge u(x,t)\}) + g(t)\right),
\end{equation}
where for $p\neq 0$,
\begin{equation}\label{As}
\As  (p)=\left(2\int_{p^\bot}P(y)\,d\H^{N-1}(y)\right)^{-1}
\end{equation}
and for a smooth set $E$, the anisotropic fractional mean curvature
at $x\in\partial E$ is given by
\begin{equation}\label{def:kappas}
-\kappa_s(x,E):=
\int_{\R^N} \frac{\chi_{E}(y)-\chi_{E^c}(y)}{\Anis (y-x)^{N+s}}\, dy
\end{equation}
(where here the ``$-$'' sign is so that convex sets
have a nonnegative curvature). Here, as in the rest of the paper, we denote with $D\cdot$ the spatial derivative. This singular integral can be given
a meaning, and shown to be finite for $C^{1,1}$ sets, see~\cite{imbert}.

Following~\cite{barlesimbert,imbert}, in order to define
the right notion of solution we need to introduce 
the following integral functionals, which
extend the definition of the curvature of the level set
of a function $v:\R^N\times [0,\infty)\to\R$:
\[
\overline I_A[v](x,t)=\int_{A} \left(\chi^+(v(y+x,t)-v(x,t))-\chi^-(v(y+x,t)-v(x,t))\right)P(y)\,dy,
\]
\[
\underline I_A[v](x,t)=\int_{A} \left(\chi_+(v(y+x,t)-v(x,t))-\chi_-(v(y+x,t)-v(x,t))\right)P(y)\,dy,
\]
with the notation $\chi^+=\chi_{[0,\infty)}$, $\chi^-=\chi_{(-\infty,0)}$, $\chi_+=\chi_{(0,\infty)}$, $\chi_-=\chi_{(-\infty,0]}$. 

\begin{remark}\label{rem:Icont}
If $\ph\in L^\infty(\R^N\times [0,\infty)) \cap C^2(B_\delta(x,t))$ for
some $(x,t)\in\R^N\times [0,\infty)$, then the functions 
$\overline I_{B_\delta(x)}[\ph]$ and $\underline I_{B_\delta(x)}[\ph]$ are pointwise continuous outside the set $\{D\ph=0\}$, in the sense that if $D\ph(x,t)\ne 0$, then
\[
\lim_{(y,s)\to(x,t)}\overline I_{B_\delta(y)}[\ph](y,s)=\overline I_{B_\delta(x)}[\ph](x,t),\; \lim_{(y,s)\to(x,t)}\underline I_{B_\delta(y)}[\ph](y,s)=\underline I_{B_\delta(x)}[\ph](x,t).
\]
If $\ph$ is just upper semicontinuous (respectively lower semicontinuous), then $\overline I_{B_\delta(\cdot)} [\ph]$ is upper semicontinuous (respectively $\underline I_{B_\delta(\cdot)}[\ph]$ is lower semicontinous). Moreover, if $\ph_k$ is a sequence of functions pointwisely converging to $\ph$, then
\[
\limsup_{k\to\infty }\overline I_A[\ph_k]\le \overline I_A[\ph],\qquad\qquad \liminf_{k\to\infty} \underline I_A[\ph_k]\ge \underline I_A[\ph]
\]
for every set $A\subseteq\R^N$.
\end{remark}
\begin{remark}\label{remI}
Observe that if $\ph\in C^2(\R^N\times [0,\infty))$ and the level set
$\{\ph(\cdot,t)=\ph(x,t)\}$ is not critical, then for any $A$,
$\overline{I}_{A}[\ph](x,t)=\underline{I}_{A}[\ph](x,t)$ and we can denote
\[
I[\ph](x,t) = \overline{I}_{\R^N}[\ph](x,t) = -\kappa_s(x,\{\ph(\cdot,t)\ge
\ph(x,t)\}).
\]
\end{remark}

\subsection{Viscosity solutions}\label{sec:visco}
The precise meaning of a solution of Equation \eqref{nlMCF} is given by one
the following equivalent definitions (see~\cite{barlesimbert,imbert}
and~\cite{chamorpon}) of \textit{viscosity solutions}: 
\begin{definition}\label{defviscoforte}
A locally bounded upper semicontinuous function $u$ is
a viscosity subsolution of \eqref{nlMCF} if for all
$\ph\in C^2(\R^N\times(0,\infty))$, at any maximum point $(x,t)$ of $u-\ph$, then
  \begin{equation}\label{eq:subsol0}
    \begin{cases}
      \partial_t\ph(x,t)\le \As (D\ph(x,t))|D\ph(x,t)|\left(-\kappa_s(x,\{\ph\ge\ph(x,t)\})+g(t)\right)
      \\ \hspace{3cm}  \text{if $D\ph(x,t)\ne0$ and $\ph(x,t)$ is not
a critical value of $\ph$,}\\
      \partial_t\ph(x,t)\le0 \quad \text{if $D\ph(x,t) = 0$}.
    \end{cases}
  \end{equation}
A locally bounded lower semicontinuous function $u$ is
a viscosity supersolution if $-u$ is a viscosity subsolution
with forcing term $-g$.
A solution is a function whose upper semicontinuous envelope is a
subsolution, while its lower semicontinuous envelope is a supersolution.
\end{definition}

\begin{definition}\label{defviscosecontinue}
A locally bounded upper semicontinuous  function $u$ is a viscosity subsolution of \eqref{nlMCF} if for all $\ph\in C^2(\R^N\times(0,\infty))$, at any
maximum point $(x,t)$ of $u-\ph$ in a ball $B_\delta(x,t)$, it holds
\begin{equation}\label{eq:subsol}
\begin{cases}
  \partial_t\ph(x,t)\le \As  (D\ph(x,t))|D\ph(x,t)|\left(\overline I_{B_\delta(x)}[\ph](x,t)+ \overline I_{\R^N\setminus B_\delta(x)}[u](x,t) +g(t)\right)
  \\ \hspace{3cm}  \text{if $D\ph(x,t)\ne0$}\\
  \partial_t\ph(x,t)\le0 \quad \text{otherwise}.
\end{cases}
\end{equation}
	
\noindent
A locally bounded lower semicontinuous function $u$ is a viscosity supersolution of \eqref{nlMCF} if for all $\ph\in C^2(\R^N\times(0,\infty))$ at any minimum point $(x,t)$ of $u-\ph$ and for any ball $B_\delta(x,t)$ it holds
\begin{equation}\label{eq:supersol}
\begin{cases}
  \partial_t\ph(x,t)\ge \As  (D\ph(x,t))|D\ph(x,t)|\left(\underline I_{B_\delta(x)}[\ph](x,t)+\underline I_{\R^N\setminus B_\delta(x)}[u](x,t) +g(t)\right)
  \\ \hspace{3cm}  \text{if $D\ph(x,t)\ne0$}\\
  \partial_t\ph(x,t)\ge0 \quad \text{otherwise}.
\end{cases}
\end{equation}

A solution is a function whose upper semicontinuous envelope is a
subsolution, while its lower semicontinuous envelope is a supersolution.
\end{definition}

Observe that in the definition above, one can take $\delta$ arbitrarily
small; moreover as usual, we may equivalently assume that the maximum (resp. minimum) points are strict.

\begin{definition}\label{def:flow}
Let $C\subset\R^N$ and $g$ a continuous function and define for $\eta>0$
$d^\eta_{C}=-\eta\vee (\eta\wedge (\dist(x,C)-\dist(x,C^c)))$,
that is $d^\eta$ is the signed distance function to $\partial C$
truncated at the levels $\pm\eta$.
We say that a family of sets $\{C(t)\}_{t>0}$ is a flow for the
geometric equation $\As^{-1}v=-\kappa_s+g$ starting from $C$ if for all $t\ge 0$, $C(t)=\{x\in\R^N:u(x,t)>0\}$,
or if for all $t\ge 0$, $C(t)=\{x\in\R^N:u(x,t)\ge 0\}$,
where $u$ solves \eqref{nlMCF} in the sense of Definition \ref{defviscosecontinue}.
\end{definition}
\begin{remark}\label{rem:geom}
It turns out that in this case, $\chi_{\{u>0\}}$ is a subsolution,
while $\chi_{\{u\ge 0\}}$ is a supersolution, in the sense of Definition \ref{defviscosecontinue}. Moreover, it is well known \cite{imbert} that the
equation in Definitions~\ref{defviscoforte} and~\ref{defviscosecontinue} is \textit{geometric}, meaning
that if we replace the initial condition with any function $u_0$ with the same
level sets $\{u>0\}$ and $\{u\ge 0\}$, the evolution $C(t)$ remains the same.
\end{remark}

Existence and comparison (uniqueness) results for evolutions defined
by the equivalent
Definitions~\ref{defviscoforte} and~\ref{defviscosecontinue} are provided
in~\cite{imbert}. It follows, as usual, that given a bounded
uniformly continuous initial data $u_0$, and denoting $u(x,t)$
the solution with $u(\cdot,0)=u_0$, then,
starting from almost all (but a countable number, at most) of the level
sets $C=\{u_0>s\}$ there exists a unique flow $C(t)=\{u(\cdot,t)>s\}$,
in the sense of Definition~\ref{def:flow}.

\section{Convergence of the discrete flows}\label{secconv}
\subsection{Main result}

The scope of this section is to prove the following result, which
is a variant of the main result in~\cite{cafsou}. The only differences
are that:
\begin{enumerate}
\item we introduce an anisotropy and a forcing term in the spirit
of~\cite{ishiipiressouganidis};
\item we simplify part of the argument, in particular when
estimating the ``mobility'' $\As(p)$;
\item we estimate in a separate subsection (Sec.~\ref{sec:finite}) the
evolution of balls,
yielding a then simpler argument to show consistency in flat regions,
or that the initial condition is not lost in the limit;
\item we give a proof (Sec.~\ref{convexityofthemobility}) of the convexity
of the mobility.
\end{enumerate}

\begin{theorem}\label{mainthm}
	Let $u_0$ be a bounded, uniformly continuous function,
 $\Om=\{u_0>0\}$, $\Om^-_t=\{x\in\R^N: u(x,t)>0\}$ and 
$\Om^+_t=\{x\in\R^N: u(x,t)\ge 0\}$,
where $u$ is a viscosity solution of \eqref{nlMCF} with initial data $u_0$.
Then 
	\[
	\begin{cases} 
	&\liminf_*u_h(x,t):=\liminf_{y\to x,nh\to t}u_h(y,nh)=1\,\, \text{in $\Om^-_t$},\\
	& \limsup^*u_h(x,t):=\limsup_{y\to x,nh\to t}u_h(y,nh)=-1\,\, \text{in $\R^N\setminus\Om^+_t$}
	\end{cases} 
	\]
\end{theorem}
\begin{remark}\label{initialdata}
	Under the assumptions of Theorem \ref{mainthm}, suppose that $u$ is such that
	$\partial \{(x,t):u(x,t)>0\}=\partial \{(x,t):u(x,t)\ge 0\}=:\Gamma_t$.
	Then we have 
	$$F_h:=\cup_n\partial\{u_h(\cdot,nh)=h^{\frac{s}{1+s}}g(nh)\}\times \{nh\}\ \to\  \cup_{t\ge0}\Gamma_t\times\{t\}\qquad \text{as }n\to\infty,
	$$ 
	in the Kuratowski sense. 
\end{remark}

\subsection{The speed of balls}\label{sec:finite}
A useful intermediate result (which illustrates why the scale in
\eqref{nuclei} is the right one) is a control on the (bounded) speed at
which balls  decrease with the discrete flow. 
Let
\begin{equation}\label{defkamax}
\kma = \max_{x\in \partial B_1} \kappa_s(x,B_1)
\end{equation}
be the maximal curvature of the unit ball. Then for any $R>0$ and $x\in\partial B_R$, by a change of variable of the form $y'=y/R$ we get that
\[
\kappa_s(x,B_R)=-\int \frac{\chi_{B_R}(y)-\chi_{B_R^c}(y)}{\Anis(y)^{N+s}}dy
= -\frac{1}{R^s}\int \frac{\chi_{B_1}(y')-\chi_{B_1^c}(y')}{\Anis(y')^{N+s}}dy'
\le \frac{\kma}{R^s}\,.
\]
This suggests that the motion of a ball should be at most
governed by $\dot{R}\sim -\kma/R^s-\|g\|_\infty$, yielding
an extinction time of order $\sim R^{1+s}$ for $R$ small.
We check now this is indeed the case for the discrete scheme.

We start with the following lemma
which estimates the speed of the scheme applied to the
unit ball:
\begin{lemma}\label{lemspeedball}
There exists $\mma$, depending only $\Anis$,
and $h_0>0$ (depending on $\Anis$ and $\|g\|_\infty$)
such that if $h<h_0$,
$P_h*(\chi_{B(0,1)}-\chi_{B(0,1)^c})\ge \|g\|_\infty h^{\frac{s}{1+s}}$ in $B(0,1-ch)$
where $c=(\kma+\|g\|_\infty)/\mma$.
\end{lemma}
\begin{proof}
We let $e=(1,0,\dots,0)$ a unit vector and denote $B=B(e,1)$.
We recall 
\[
P_h(x)=\frac{1}{h^{N/(1+s)}}
\frac{1 }{1+\Anis \left(\frac{x}{h^{1/(1+s)}}\right)^{N+s}}
=h^{\frac{s}{1+s}} \frac{1}{h^{\frac{N+s}{1+s}}+\Anis (x)^{N+s}},
\]
so that
\[
\lim_{h\to 0} h^{\frac{-s}{1+s}} P_h*(\chi_B-\chi_{B^c})(0)=-\kappa_s(0,\partial B)
\]
More precisely, it stems from the convexity of $B$ that this limit
is, in fact, an infimum. Indeed,
using the symmetry of $\Anis$, one sees that the convolution
$ h^{-\frac{s}{1+s}} P_h*(\chi_B-\chi_{B^c})(0)$ is given by
\[
-\int_{\R^N\setminus (B\cup(-B))}
\frac{dx}{h^{\frac{N+s}{1+s}}+\Anis (x)^{N+s}}
\]
which is monotone in $h$, and it follows
\begin{equation}\label{eq:estimkh}
h^{-\frac{s}{1+s}} P_h*(\chi_B-\chi_{B^c})(0) \ge -\kma
\end{equation}
for any $h>0$.

Then, we need to estimate $D P_h*(\chi_{B}-\chi_{B^c})$ near
$\partial B$. In fact it is enough to have an estimate for
$x=\bar te$ with $|\bar t|\lesssim c(h)$ for some $c(h)\gg h$.
A simple analysis shows that the scaling of $\bar t$ should
be between $h$ and $h^{1/(1+s)}\gg h$, in what follows we
therefore consider $\bar t=\tau h^{\frac{1+s/2}{1+s}}$ for $-1\le \tau \le 1$.
One has
\[
G:=h^{\frac{-s}{1+s}} e\cdot D P_h*(\chi_B-\chi_{B^c})(\bar te)
=2\int_{\partial B} \frac{-\nu_B\cdot e}{h^{\frac{N+s}{1+s}}+\Anis (\bar te-x)^{N+s}}d\H^{N-1}(x).
\]
With the change of variable $x=y h^{1/(1+s)}$, we have, denoting
$h^{-1/(1+s)}B$ the ball $B(h^{-1/(1+s)}e,h^{-1/(1+s)})$,
\begin{multline*}
G=-2\int_{\partial (h^{\frac{-1}{1+s}}B)}
\frac{h^{\frac{N-1}{1+s}} \nu_B\cdot e}{h^{\frac{N+s}{1+s}}(1+\Anis (\tau h^{\frac{s/2}{1+s}} e-y)^{N+s})}d\H^{N-1}(y)
\\
=-\frac{2}{h}\int_{\partial (h^{\frac{-1}{1+s}}B)}
\frac{ \nu_B\cdot e}{1+\Anis (\tau h^{\frac{s/2}{1+s}} e-y)^{N+s}}d\H^{N-1}(y)
\end{multline*}
Given $R>0$ let $B^h:= B(0,Rh^{-(N-1)/((N+s)(1+s))})$: first observe that
if $y\not\in B^h$, 
\[
|\tau h^{\frac{s/2}{1+s}} e-y| \ge
Rh^{-\frac{N-1}{(N+s)(1+s)}} - h^{\frac{s/2}{1+s}} 
\ge \frac{R}{2}h^{-\frac{N-1}{(N+s)(1+s)}}
\]
if $h$ is small enough, so that
\begin{multline*}
\left|2\int_{\partial (h^{\frac{-1}{1+s}}B)\setminus B^h}
\frac{ \nu_B\cdot e}{1+\Anis (\tau h^{\frac{s/2}{1+s}} e-y)^{N+s}}d\H^{N-1}(y)
\right|
\\ \le h^{-\frac{N-1}{1+s}}\frac{1}{1+\underline{c}\big(\frac{R}{2}h^{-\frac{N-1}{(N+s)(1+s)}}\big)^{N+s}}
\le \frac{1}{h^{\frac{N-1}{1+s}} +\underline{c}\big(\frac{R}{2}\big)^{N+s}}
\end{multline*}
which can be made arbitrarily small by choosing  $R$ large enough.
On the other hand, if $y\in B^h\cap \partial (h^{-1/(1+s)}B)$, as
$h^{-1/(1+s)}\gg  Rh^{-(N-1)/((N+s)(1+s))}$ (indeed
$1/(1+s)-(N-1)/((N+s)(1+s))=(N+s-N+1)/((N+s)(1+s))=1/(N+s)>0$),
if $h$ is small enough one has $\nu_B\cdot e\le -1/2$, which yields
\begin{multline*}
-2\int_{\partial (h^{\frac{-1}{1+s}}B)\cap B^h}
\frac{ \nu_B\cdot e}{1+\Anis (\tau h^{\frac{s/2}{1+s}} e-y)^{N+s}}d\H^{N-1}(y)
\\\ge \int_{\partial (h^{\frac{-1}{1+s}}B)\cap B^h}
\frac{d\H^{N-1}(y)}{1+\overline{c}\big|\tau h^{\frac{s/2}{1+s}} e-y\big|^{N+s}}
\to \int_{\{y\cdot e=0\}} \frac{d\H^{N-1}(y)}{1+\overline{c}|y|^{N+s}}=: 2\mma>0
\end{multline*}
as $h\to 0$. All in all, with an appropriate choice of $R$, we
see that there exists $h_0$ (depending on $\Anis$ and $\|g\|_\infty$ but not, in fact,
on the particular point we have chosen on $\partial B$)
such that if $h<h_0$, one has, recalling~\eqref{eq:estimkh}
\begin{align*}
 &h^{\frac{-s}{1+s}} P_h*(\chi_B-\chi_{B^c})(0) \ge -\kma
\\
& h^{\frac{-s}{1+s}} e\cdot D P_h*(\chi_B-\chi_{B^c})(\bar t e) \ge \frac{\mma}{h}
\end{align*}
for $|\bar t|\le  h^{\frac{1+s/2}{1+s}}$. Hence if $-\kma+\bar{t}\mma/h \ge
\|g\|_\infty$ and $|\bar t|\le h^{\frac{1+s/2}{1+s}}=h\cdot h^{-(s/2)/b(1+s)}$, we have
$P_h(\chi_B-\chi_{B^c})(\bar t e)\ge h^{-s/(1+s)}\|g\|_\infty$.
Choosing $\bar t=h (\|g\|_\infty+\kma)/\mma$, and possibly reducing $h_0$,
we check that indeed $|\bar t|\le h^{\frac{1+s/2}{1+s}}$ and the inequality
holds. We will show in Corollary~\ref{convexity}
that the level sets of $P_h(\chi_B-\chi_{B^c})$ are all convex,
so that the thesis of the Lemma holds. 
\end{proof}

Thanks to a simple scaling argument, we find that given $R>0$, if $h$ is small
enough,
$P_h*(\chi_{B(0,R)}-\chi_{B(0,R)^c})\ge \|g\|_\infty h^{\frac{s}{1+s}}$ in $B(0,R-ch)$
with now $c=(\kma/R^s+\|g\|_\infty)/\mma$. As a result, we
have the following corollary:
\begin{corollary}\label{corboundedspeed}
If $E_0=B(x_0,R)$, then for $h$ small enough,
$u_h(\cdot,nh)\ge \widetilde\chi_{B(x_0,R/2)}$ as long as 
\begin{equation}\label{nuvole}
nh\le \frac{\mma R^{1+s}}{2(R^s\|g\|_\infty+2^s\kma)} .
\end{equation}
In particular, if $R$ is small, one has $u_h(x_0,nh)=1$
for $nh\lesssim R^{1+s}$.
\end{corollary}

\subsection{The consistency result}
The main difficulty to prove convergence is a consistency result. The 
strategy of the proof follows~\cite{bargeo,cafsou}, with some
slight simplification.
The important point is to show that
$\overline u:= \limsup^*u_h$  and $\underline u:=\liminf_* u_h$ are respectively  viscosity supersolution and subsolution of \eqref{nlMCF}. 
\begin{proposition}\label{consistency}
	The functions $\liminf_*u_h(x,t)$ and $\limsup^*u_h(x,t)$ defined in Theorem \ref{mainthm} are, respectively, a supersolution and a subsolution for \eqref{nlMCF}.
\end{proposition}
Once this consistency result is settled down, the proof of Theorem \ref{mainthm} easily follows: we first notice that $\liminf_*u_h(x,t)$ and $\limsup_*u_h(x,t)$ take only values in $\{\pm 1\}$. Thus to conclude we only have to recall (see \cite{barsonsou}) that the maximal upper semicontinuous  subsolution and minimal upper semicontinuous  supersolution of \eqref{nlMCF} are given by $\chi_{(-\infty,0]}(u)-\chi_{(0,\infty)}(u)$ and $\chi_{(-\infty,0)}(u)-\chi_{[0,\infty)}(u)$, where $u$ is a solution of \eqref{nlMCF}. 
The fact that the initial data is taken easily follows by comparison,
using the results of Section~\ref{sec:finite}.
This immediately entails the statement of Theorem \ref{mainthm}.

We pass now to the  proof of Proposition \ref{consistency}.
\begin{proof}[Proof of Proposition \ref{consistency}]
A first observation is that, as an easy consequence of Corollary~\ref{corboundedspeed}, the functions $\overline u(x,0)=\widetilde\chi_{\overline\Omega}(x)$ and $\underline u(x,0)=\widetilde\chi_{\overline\Omega}(x)$, in other words, they satisfy the required initial data. 

Let us fix $(x_0,t_0)\in\R^N\times (0,+\infty)$ and $\ph$ and assume that $(x_0,t_0)$ is a point of maximum of $u-\ph$.
Since $\overline u$ takes values $\pm 1$ and it is upper semicontinuous , if $\overline u(x_0,t_0)=-1$ then it is constant in a neighborhood of $(x_0,t_0)$ and thus $|D\overline u|=|D\ph|=|\partial_t\ph|=0$ and  so \eqref{eq:subsol} trivially holds. The same assertion holds if $(x_0,t_0)$ is an interior point of $\{\overline u=1\}$. So we can suppose that $\overline u(x_0,t_0)=1$ and $(x_0,t_0)$ is a boundary point of $\{\overline u=1\}$. In this case, replacing first $\ph$ with
$\ph_\eta(x,t)+\eta(|x-x_0|^2+|t-t_0|^2)$, we define for each $h>0$ small
enough the point
	\begin{equation}\label{max}
	(x_h,n_hh)={\rm argmax}_{\R^N\times\mathbb N}\,\, u_h^*-\ph_\eta.
	\end{equation}
The main inequality  will be proved for the function $\ph_\eta$, however
one can easily show (thanks to Remark~\ref{rem:Icont}) that
it then follows for $\ph$ when one sends $\eta\to 0$,
hence in the sequel we will drop the index $\eta$ and write simply $\ph$.

	Up to passing to a (not relabeled) subsequence, we can suppose that $(x_h,n_hh)$ converge to a point $(x_1,t_1)$. In this case, by the regularity of $\ph$ and since $\overline u$ is upper semicontinuous  we have that
	\[
	\overline u(x_1,t_1)-\ph(x_1,t_1)\ge\limsup_{h\to0} u^*(x_h,n_hh)-\ph(x_h,n_hh)\ge \overline u(x_0,t_0)-\ph(x_0,t_0)
	\]
	and thus $(x_1,t_1)=(x_0,t_0)$ since the latter is a strict global maximum of $u-\ph$. Here we denoted by $j^*$ the upper semicontinuous  envelope of a function $j$.
	
	By the definition of the points $(x_h,n_hh)$ and since $u_h^*\in\{\pm1\}$ it is easy to show that
	\[
	u_h^*(x,nh)\le {\rm  sign}^*(\ph(x,n_hh)-\ph(x_h,n_hh)). 
	\]
	We recall now that 
	\[
	u_h(\cdot,n_hh)=\widetilde\chi_{\{P_h*u_h(\cdot,(n_h-1)h)>g((n_h-1)h) h^\frac{s}{1+s}\}}\le \widetilde\chi_{\{P_h*u_h^*(\cdot,(n_h-1)h)>g((n_h-1)h) h^\frac{s}{1+s}\}}
	\]
	so that (since the right-hand side of the previous inequality is upper semicontinuous)
	\[
	u^*_h(\cdot,n_hh)\le \widetilde\chi_{\{P_h*u_h^*(\cdot,(n_h-1)h)>g((n_h-1)h) h^\frac{s}{1+s}\}}.
	\]
	By computing the previous inequality in $x=x_h$, where $u_h^*(\cdot,n_h h)$ takes the value $1$, we obtain
	\[
	1=u^*_h(x_h,n_hh)\le \widetilde\chi_{\{P_h*u_h^*(x_h,(n_h-1)h)>g((n_h-1)h) h^\frac{s}{1+s}\}}\le 1,
	\]
	that is,
	\[
	0\le P_h*u_h^*(\cdot,(n_h-1)h)(x_h)-g((n_h-1)h) h^\frac{s}{1+s}.
	\]
	Since $u_h^*=\pm 1$ and $u_h^*(x_h,n_hh)=1$, the previous inequality can be written as
	\begin{equation}\label{keyinequality}
	\begin{aligned} 
	0\le\int_{\R^N} \big(&\chi^+(u_h^*(y+x_h,(n_h-1)h)-u_h^*(x_h,n_hh))\\
	&-\chi^-((u_h^*(y+x_h,(n_h-1)h))-u_h^*(x_h,n_hh))\big)P_h(y)dy-g((n_h-1)h)h^{\frac{s}{1+s}}.
	\end{aligned}
	\end{equation}
	\noindent
	The idea is now to estimate the right-hand side of \eqref{keyinequality} by means of several terms, converging to the difference between the right-hand side of \eqref{eq:supersol} and its left-hand side. We define $t_h=n_hh$, $\ph_h(y,s)=\ph(x_h+y,s)-\ph(x_h,t_h)$.
By \eqref{max} we get that for every $A\subseteq\R^N$ it holds
	\begin{equation}\label{cont2}
	\{y:u_h^*(y+x_h,t_h-h)\ge u_h^*(x_h,t_h)\}\cap A\subseteq\{y:\ph_h(y,t_h-h)\ge0\}\cap A,
      \end{equation}
      and
	\begin{equation}\label{cont2bis}
	\{y:\ph_h(y,t_h-h)<0\}\cap A\subseteq \{y:u_h^*(y+x_h,t_h-h)< u_h^*(x_h,t_h)\}\cap A.
	\end{equation}
	Given a small number $\gamma>0$, if $\delta$ is small enough, then
	\begin{equation}\label{upbdph}
\begin{aligned}	&\ph_h(\cdot,t_h-h)\le \ph_h(\cdot,t_h)-(\partial_t\ph(x_0,t_0)-\gamma)h
\\ &
|D \ph(x_0,t_0)|-\gamma \le |D\ph(\cdot,t_h)| \le
|D \ph(x_0,t_0)|+\gamma
\end{aligned}
\end{equation}
      inside $B_\delta(0)$.

The analysis is now split into two main parts, depending whether $|D\ph(x_0,t_0)|=0$ or not. 
	\medskip
	
\paragraph{\bf Step 1} Case $|D\ph(x_0,t_0)|\ne0$.
 We begin by writing \eqref{keyinequality} as
\begin{equation}\label{keyineq2}
0\le \textit{I}+h^{\frac{s}{1+s}} g((n_h-1)h)
\end{equation}
and then
\begin{equation}\label{split1}
0\le \textit{I}+h^{\frac{s}{1+s}}g(t_h-h)\le \textit{II}+\textit{III}+h^{\frac{s}{1+s}}g(t_h-h)
\end{equation}
where $\textit{II}$ and $\textit{III}$, implicitly depending on $h$ and on $\delta$,  are given by 
\begin{equation}\label{II}
\begin{aligned} 
\textit{II}=\int_{B_\delta(0)^c} \big(&\chi^+(u_h^*(y+x_h,t_h-h)-u_h^*(x_h,t_h))\\
&-\chi^-(u_h^*(y+x_h,t_h-h)-u_h^*(x_h,t_h))\big)P_h(y)\,dy
\end{aligned} 
\end{equation}
	
\begin{equation}\label{III}
\textit{III}=\int_{B_\delta(0)} 
\big(\chi^+(\ph_h(y,t_h-h))
-\chi^-(\ph_h(y,t_h-h))\big)P_h(y)\,dy.
\end{equation}
The inequality \eqref{split1} follows from the fact that $\chi^+-\chi^-$
is a non-decreasing function and by \eqref{cont2}-\eqref{cont2bis}
(with $A=B_\delta(0)$).
	
We claim that 
\begin{equation}\label{limsup1}
  \limsup_{h\to0} h^{-{\frac{s}{1+s}}}\textit{II}\le \overline I_{B_\delta(x_0)^c}[u](x_0,t_0).
\end{equation}
Indeed by the definition of $P_h$, we have that
$P_h/\sigma_h$ converges (in $L^1((B_\delta(0))^c)$ as well as 
$L^\infty((B_\delta(0))^c)$) to the anisotropic fractional kernel $\Anis^{-(N+s)}$.
This, together with the fact that $\chi^+-\chi^-$ is u.s.c., and the fact that $\limsup_* u_h^*-u_h^*(x_h,t_h)=\bar u-\bar u(x_0,t_0)$, implies \eqref{limsup1}.

Let us divide again $\textit{III}$ as
\begin{equation}\label{split2}
  \textit{III}\le \textit{IV}+2\textit{V}
\end{equation}
with
\begin{equation}\label{IV}
  \textit{IV}=
  \int_{B_\delta(0)} \big( \chi^+(\ph_h(y,t_h))-\chi^-(\ph_h(y,t_h))\big) P_h(y)\,dy
\end{equation}
and
\begin{equation}\label{Vpm}
  \begin{aligned} 
    \textit{V} &= \int_{B_\delta(0)}\left(\chi^+(\ph_h(y,t_h)-(\partial_t\ph(x_0,t_0)-\gamma)h)-\chi^+(\ph_h(y,t_h))\right)P_h(y)\,dy\\
    & = -\int_{B_\delta(0)}\left(\chi^-(\ph_h(y,t_h)-(\partial_t\ph(x_0,t_0)-\gamma)h)-\chi^-(\ph_h(y,t_h))\right)P_h(y)\,dy
  \end{aligned} 
\end{equation}
(using $\chi^++\chi^-=1$).
It is immediate to see that \eqref{split2} follows by adding and subtracting the integral defined in \eqref{Vpm} and using~\eqref{upbdph}. 
	
We aim to prove  now that 
\begin{equation}\label{limsup2}
  \limsup_{h\to0}h^{-\frac{s}{1+s}}\textit{IV}\le\overline I_{B_\delta(x_0,t_0)}[\ph](x_0,t_0),
\end{equation}
and
\begin{equation}\label{limsup3}
  \limsup_{h\to0}h^{-\frac{s}{1+s}}\textit{V}\le  |D\ph(x_0,t_0)|^{-1}\As^{-1}(D\ph(x_0,t_0))(\gamma-\partial_t\ph(x_0,t_0)).
\end{equation}

The proof of those two latter statements is slightly more involved. 
To prove the first, we begin as in~\cite{cafsou} with the following simple lemma.
\begin{lemma}\label{A1}
If $|D\ph(x_0,t_0)|\ne 0$,
there exists a constant $C$ such that for any $r>0$, it holds
\begin{equation}\label{app1}
\int_{\partial B_r} \left(\chi^+(\ph_h(y,t_h))-\chi^-(\ph_h(y,t_h))\right)
d\H^{N-1}(y)\le Cr^{N}.
\end{equation}
\end{lemma}

\begin{proof}[Proof of Lemma \ref{A1}]
By the trivial estimate
\[
\int_{\partial B_r} \left(\chi^+(\ph_h(y,t_h))-\chi^-(\ph_h(y,t_h)\right)\,d\H^{N-1}(y)\le 2r^{N-1},
\]
is clear that we have to show the statement of the lemma only for $r$ small. Up to a rotation of the coordinates, we can suppose that $D \ph (x_h,t_h)/|D \ph (x_h,t_h)|=e_1$. Since $\ph$ is a regular function we have that in a sufficiently small neighborhood of $0$
\[
|\ph_h(y,t_h)| \le |D\ph(x_y,t_h)|y_1+ |D^2\ph(x_h,t_h)|\,|y|^2.
\]
This implies that the set of integration in the left-hand side of \eqref{app1} is contained in the set $\{y=(y_1,y')\in\partial B_r: |y_1|\le c|y'|^2 \}$ whose measure can be estimated by $cr^{N-2}r^2\sim r^{N}$ (indeed,
on the complement, the part where $\ph_h>0$ compensates exactly 
the part where $\ph_h<0$). Hence~\eqref{app1} holds.
Notice that if $r$ and $h$ are small enough, the constant
in~\eqref{app1} depends only on
(the dimension and) $D\ph(x_0,t_0)$, $D^2\ph(x_0,t_0)$.
\end{proof}

\noindent	As a consequence we get that inequality \eqref{limsup2} holds true. Indeed		by using polar coordinates we get that 
		\begin{equation}\label{coro1A}
		|\sigma_h^{-1}\textit{IV}|\le C_N\int_0^\delta (\sigma_h^2+r^2)^{-\frac{N+s}{2}}U_h(r)\,dr
		\end{equation}
		with
		\[
U_h(r)=\int_{\partial B_r} \left(\chi^+(\ph_h(y))-\chi^-(\ph_h(y))\right)\,d\H^{N-1}(y).
		\]
		By means the previous lemma, we get that $U_h(r)\le Cr^N$ and we can apply Lebesgue's Convergence Theorem and conclude that the (superior) limit of $\sigma_h^{-1}\textit{III}$ is exactly the right-hand side of \eqref{limsup2}.   
	
We pass to the proof of \eqref{limsup3}. We only follow
partially~\cite{cafsou} for this estimate. Let us first observe that
\begin{multline*}
\chi^+(\ph_h(y,t_h)-(\partial_t\ph(x_0,t_0)-\gamma)h)-\chi^+(\ph(y,t_h))
\\ = \begin{cases} \ds -\chi_{\{0\le \ph(\cdot,t_h)< (\partial_t\ph-\gamma)h\}}(y)
 & \textup{if } \partial_t\ph-\gamma\ge 0\\
\ds \chi_{\{(\partial_t\ph-\gamma)h\le \ph(\cdot,t_h)<0 \}}(y) 
& \textup{if } \partial_t\ph-\gamma\le 0.
\end{cases}
\end{multline*}
Assuming for instance that $\partial_t\ph-\gamma> 0$, we obtain
(denoting simply by $B_r$ the ball $B_r(0)$)
\[
\textit{V}=-\int_{\{0\le \ph(\cdot,t_h)< (\partial_t\ph-\gamma)h\}\cap B_\delta}P_h(y)\,dy
=\int_{\{0\le \ph(\sigma_h^{1/s}\cdot,t_h)< (\partial_t\ph-\gamma)h\}
\cap B_{\delta/\sigma_h^{1/s}}}P(y)\,dy
\]
hence (using the co-area formula)
\begin{align*}
\textit{V} &= -\int_0^{(\partial_t\ph-\gamma)h}d\tau
 \int_{\partial \{\ph(\sigma_h^{1/s}\cdot,t_h)\ge \tau\}\cap B_{\delta/\sigma_h^{1/s}}}
\frac{ P(y)}{\sigma_h^{1/s}|D\ph(\sigma_h^{1/s}y,t_h)|}d\H^{N-1}(y)\\
& \le -\frac{h}{\sigma_h^{1/s}}\frac{\partial_t\ph(x_0,t_0)-\gamma}{|D\ph(x_0,t_0)|+\gamma}\int_0^1 d\tau \int_{\partial \{\ph(\sigma_h^{1/s}\cdot,t_h)\ge h\alpha\tau\}\cap B_{\delta/\sigma_h^{1/s}}}P(y)d\H^{N-1}(y)
\end{align*}
where we have denoted, for short, $\alpha=\partial_t\ph-\gamma$, and used again~\eqref{upbdph}. To sum up,
\begin{equation}\label{estimVp}
\sigma_h^{-1}\textit{V} \le 
-\frac{\partial_t\ph(x_0,t_0)-\gamma}{|D\ph(x_0,t_0)|+\gamma}\int_0^1 d\tau \int_{\partial \{\ph(\sigma_h^{1/s}\cdot,t_h)\ge h\alpha\tau\}\cap B_{\delta/\sigma_h^{1/s}}}P(y)d\H^{N-1}(y)
\end{equation}

Possibly reducing $\delta$, and assuming that the $N$th coordinate
is along the vector $e_N=D\ph(x_0,t_0)/|D\ph(x_0,t_0)|$, for $h$ small
enough we can represent the level surface $\partial \{\ph_h(\cdot,t_h)\ge  h\ell\}$ by a graph $\{ (y',f_{\ell,h}(y'))\,:\, y'\in B'_{\delta}\}\cap B_\delta$ (where
$B'_r$ denotes the $(N-1)$-dimensional ball in $e_N^\perp$ with center $0$
and radius $r$), with $f_{\ell,h}$ which goes uniformly
(in both $y'$ and $\ell$), in $C^2$ norm as $h\to 0$,
to the function $f_0$ representing in the same way 
the surface $\partial \{ y\in B_\delta:\ph(x_0+y,t_0)\ge \ph(x_0,t_0)\}$,
and which is such that $D'f_0(0)=0$.
We observe moreover that $f_{\ell,h}(0) = h\ell/|D\ph(x_0,t_0)| + o(h)$.
We denote $D_{\ell,h}\subseteq B'_{\delta}$ the set of points $y'$ such
that $(y',f_{\ell,h}(y'))\in B_\delta$.

Now, we have that
\begin{multline*}
\int_{\partial \{\ph(\sigma_h^{1/s}\cdot,t_h)\ge h\alpha\tau\}\cap B_{\sigma_h^{-1/s}\delta}}P(y)d\H^{N-1}(y)
\\= \int_{\sigma_h^{-1/s}D_{\alpha\tau,h}} P(y',\sigma_h^{-1/s}f_{\alpha\tau,h}(\sigma_h^{1/s}y'))
\sqrt{1+|D'f_{\alpha\tau,h}(\sigma_h^{1/s} y')|^2}dy'.
\end{multline*}
Given $R>0$, we split the last integral into an integral in $B'_R$ and
an integral in $\sigma_h^{-1/s}D_{\alpha\tau,h}\setminus B'_R$: clearly
the latter is controlled uniformly by $cR^{-1-s}$ as the gradients
of the functions $f_{\alpha\tau,h}$ are uniformly bounded. We now try
to express the limit, as $h\to0$, of 
\[
\int_{B'_R} P(y',\sigma_h^{-1/s}f_{\alpha\tau,h}(\sigma_h^{1/s}y'))
\sqrt{1+|D'f_{\alpha\tau,h}(\sigma_h^{1/s}y')|^2}dy'.
\]
Observe that if $|y'|\le R$ and $h$ is small enough,
\[
\sigma_h^{-1/s}f_{\alpha\tau,h}(\sigma_h^{1/s}y')
\approx \frac{\alpha\tau h}{\sigma_h^{1/s}|D\ph(x_0,t_0)|}+|D'f_{\alpha\tau,h}(\theta_h y')||y'|
\]
for some $\theta_h\in [0,\sigma_h^{1/s}]$. Using $h/\sigma_h^{1/s}=\sigma_h$,
we find that $\sigma_h^{-1/s}f_{\alpha\tau,h}(\sigma_h^{1/s}y')\to 0$
uniformly (in $\tau$ and $y'\in B(0,R')$) as $h\to 0$.
We deduce (using again that $D'f_{\alpha\tau,h}(\sigma_h^{1/s}y')\to 0$)
\[
\int_{B'_R} P(y',\sigma_h^{-1/s}f_{\alpha\tau,h}(\sigma_h^{1/s}y'))
\sqrt{1+|D'f_{\alpha\tau,h}(\sigma_h^{1/s}y')|^2}dy'\to 
\int_{B'_R} P(y',0)dy'
\]
uniformly in $\tau$ as $h\to 0$. Hence returning to~\eqref{estimVp},
we find, for any $R$,
\[
\limsup_{h\to 0}\sigma_h^{-1}\textit{V} \le
-\frac{\partial_t\ph(x_0,t_0)-\gamma}{|D\ph(x_0,t_0)|+\gamma}\int_0^1 d\tau
\left(\int_{B'_R} P(y',0)dy' +\frac{c}{R^{1+s}}  \right)
\]
and sending $R\to\infty$ we deduce
\[
\limsup_{h\to 0}\sigma_h^{-1}\textit{V} \le
-\frac{\partial_t\ph(x_0,t_0)-\gamma}{|D\ph(x_0,t_0)|+\gamma}
\int_{D\ph(x_0,t_0)^\perp} P(y)d\H^{N-1}(y).
\]
On the other hand, if $\partial_t\ph(x_0,t_0)-\gamma< 0$, the same
proof will show
\[
\limsup_{h\to 0}\sigma_h^{-1}\textit{V} \le
-\frac{\partial_t\ph(x_0,t_0)-\gamma}{|D\ph(x_0,t_0)|-\gamma}
\int_{D\ph(x_0,t_0)^\perp} P(y)d\H^{N-1}(y).
\]
Together with~\eqref{limsup2} we deduce
\[
\limsup_{h\to 0} \sigma_h^{-1} \textit{III}\le
\overline I_{B_\delta(x_0,t_0)}[\ph](x_0,t_0)
-2\frac{\partial_t\ph(x_0,t_0)-\gamma}{|D\ph(x_0,t_0)|\pm\gamma}
\int_{D\ph(x_0,t_0)^\perp} P(y)d\H^{N-1}(y)
\]
and using~\eqref{limsup1}, we get
\begin{multline*}
\limsup_{h\to 0}\sigma_h^{-1}\textit{I}
\\\le
\overline I_{B_\delta(x_0,t_0)^c}[u](x_0,t_0)
+\overline I_{B_\delta(x_0,t_0)}[\ph](x_0,t_0)
-\frac{\partial_t\ph(x_0,t_0)-\gamma}{|D\ph(x_0,t_0)|\pm\gamma}
\As(D\ph(x_0,t_0))^{-1},
\end{multline*}
where $\As$ is defined in~\eqref{As}. Since $g$ is continuous,
$g(t_h-h)\to g(t_0)$ as $h\to 0$ and we obtain~\eqref{eq:subsol}
by taking the limsup in~\eqref{keyineq2} and
sending then $\gamma\to 0$.
\medskip

\paragraph{\bf Step 2} Case $|D\ph(x_0,t_0)|=0$. 
A first classical observation (see, for instance,~\cite{bargeo}) is that in this case
one can ``decouple'' the test function $\ph$ as the sum of a function of $x$
and a function of $t$. Indeed, since $(x_0,t_0)$ is a critical point
of $t\mapsto \ph(x,t)-\partial_t\ph(x_0,t_0)(t-t_0)$,  there
exists $\mu>0$ such
that, near $(x_0,t_0)$, $\ph(x,t)\le
\psi(x,t):= \ph(x_0,t_0)+\partial_t\ph(x_0,t_0)(t-t_0)
+ \mu(|x-x_0|^2+|t-t_0|^2)/2$. Then as before, $(x_0,t_0)$ is a strict local
maximum of $u^*-\psi$. For $n\ge 1$,
let $(x_n,t_n)$, $t_n<t_0$, be a local maximum of
$u^*(x,t)-\psi(x,t)-1/(n(t_0-t))$. Such a maximum exists since for any $x$, $u^*(x,t)-\psi(x,t)-1/(n(t_0-t))$ diverges to $-\infty$ as $t\to t_0^+$ and $t\to-\infty$. Moreover is easy to see that   $(x_n,t_n)\to(x_0,t_0)$ as $n\to\infty$. Assume first that
\[
D\psi(x_n,t_n)= \mu(x_n-x_0)\neq 0
\]
for infinitely many $n$. Then we have, thanks to the Step $1$,
\begin{multline}\label{RAA1}
\partial_t\ph(x_n,t_n)+\mu(t_n-t_0) +\frac{1}{n(t_0-t_n)^2}
\le \mu|x_n-x_0|\As(x_n-x_0)(I[|x-x_0|^2/2](x_n)+g(t))
\\
\sim C\mu |x_n-x_0| (C|x_n-x_0|^{-s}+\|g\|_\infty)\to 0
\end{multline}
as $n\to\infty$, since the $s$-curvature of a ball of radius
$r$ is of order $r^{-s}$. We deduce~\eqref{eq:subsol}.

Hence we are reduced to the case where $D\psi(x_n,t_n)=0$ for
all $n$ large enough, which implies $x_n=x_0$. We argue by contradiction supposing that $\partial\ph_t(x_0,t_0)>0$.
For any $x$ and $t<t_n$, by
maximality of $(x_n,t_n)$ for the function $u^*(x,t)-\psi(x,t)+1/n(t-t_n)$ we have
\[
\begin{aligned} 
u^*&(x,t)-\ph(x_0,t_0)-\partial_t\ph(x_0,t_0)(t-t_0)-\frac\mu2(|t-t_0|^2+|x-x_0|^2)-\frac{1}{n(t_0-t)}\\
&\le 
u^*(x_0,t_n)-\ph(x_0,t_0)-\partial_t\ph(x_0,t_0)(t_n-t_0)-\frac\mu2|t_n-t_0|^2-\frac{1}{n(t_0-t_n)}
\end{aligned} 
\]
so that
\[
\begin{aligned} 
u^*(x,&t) -u^*(x_0,t_n)\\
& \le -\partial_t\ph(x_0,t_0)(t_n-t) - \mu(t_n-t)(t-t_0)
-\frac{1}{n(t_0-t_n)}+\frac{1}{n(t_0-t)} + \frac{\mu}{2}|x-x_0|^2
\\
& \le \left(-\partial_t\ph(x_0,t_0) + \mu(t_0-t)\right)(t_n-t) + \frac{\mu}{2}|x-x_0|^2.
\end{aligned} 
\]
It follows that $u^*(x,t)=-1$
 if $|x-x_0| < r(t):=\sqrt{\partial_t\ph(x_0,t_0)(t_n-t)/\mu}$,
provided $n$ is large enough and $t$ is close enough to $t_n$
so that $-\partial_t\phi(x_0,t_0) + \mu(t_0-t) \le -\partial_t\ph(x_0,t_0)/2$.
It also follows that $u^*(x_0,t_n)=1$.

Now Corollary~\ref{corboundedspeed}  yields that $u^*(x_0,t+\tau)=-1$ 
for $\tau\le \gamma/(4\bar k)r(t)^{1+s}\sim (t_n-t)^{\frac{1+s}{2}}$, so that if $t$
is close enough to $t_n$ this is true up to $\tau=t_n-t$, a contradiction.

\end{proof} 

\subsection{Convexity of the mobility}\label{convexityofthemobility}

The equation which is solved in the limit can be written as
\[
\partial_t u = \Phi(Du)\left( -\kappa_s(x,\{u\ge u(x)\}) + g(t)\right),
\]
with $\Phi(p)$, the (inverse of the) \textit{mobility}, is the
one-homogeneous function $|p| \As(p)$ (which from now on  we will often denote by $\Phi$).
If $\Phi$ is convex, this law
precisely states that the boundary of the level sets of $u$ evolve
with the speed $-\kappa_s+g$, where $-\kappa_s$ is the fractional
curvature and $g$ the forcing term, in the direction of the
``$\Phi$-normal'' $\partial\Phio(\nu)$, where $\nu$ is the normal
to the level set, and $\Phio$ the polar of $\Phi$ (or dual norm).

We will show that $\Phi$ is indeed a convex (and obviously
even, one-homogeneous) function, hence a norm.
\begin{lemma}
The $1$-homogeneous function
\[
\Phi(p) := |p|\left(2\int_{p^\perp}  \frac{dx}{1+\Anis(x)^{N+s}} \right)^{-1}
\]
is a convex function in~$\R^N$.
\end{lemma}
\begin{proof}
Consider $p^0,p^1\in\R^N$: without loss of generality, we
assume that $\vect\{p^0,p^1\}=\vect\{e_1,e_2\}$ (where $(e_i)_{i=1}^N$
is the canonical basis). We denote $x=(x_1,x_2,x')\in\R^N$,
and for $p=(p_1,p_2,0)\in\vect\{p^0,p^1\}$, $R p = (-p_2,p_1,0)$.
One can check that for such a $p$,
\[
\Phi(p) = \left(\int_{\R^{N-2}} dx' \int_{-\infty}^{+\infty} dz \frac{1}{a+\Anis(x'+zRp)^{N+s}} \right)^{-1}.
\]
Then one computes, performing successively the changes of
variables $\xi'=x'/z$ and $t=\Anis(\xi'+Rp)z$,
\begin{align*}
\Phi(p) &= \left(\int_{\R^{N-2}} dx' \int_{-\infty}^{+\infty} dz \frac{1}{1+z^{N+s}\Anis(\frac{x'}{z}+Rp)^{N+s}} \right)^{-1}
\\ &
= \left(\int_{\R^{N-2}} d\xi' \int_{-\infty}^{+\infty} dz \frac{z^{N-2}}{1+z^{N+s}\Anis(\xi'+Rp)^{N+s}} \right)^{-1}
\\ &
= \left(\int_{\R^{N-2}} \frac{d\xi'}{\Anis(\xi'+Rp)^{N-1}}\int_{-\infty}^{+\infty}  dt \frac{t^{N-2}}{1+t^{N+s}} \right)^{-1},
\end{align*}
so that
\begin{equation}\label{eq:Phi}
\Phi(p) =  C(N,s)
\left(\int_{\R^{N-2}} \frac{d\xi'}{\Anis(\xi'+Rp)^{N-1}} \right)^{-1}
\end{equation}
for some constant $C(N,s)$.
Now, consider $\lambda\in [0,1]$ and the functions
(assuming $p_1,p_0,\lambda p_1+(1-\lambda)p_0\neq 0$)
\begin{align*}
& h(x') = \frac{1}{\Anis(x'+R(\lambda p_1+(1-\lambda) p_0))^{N-1}},
\\ & f(x') = \frac{1}{\Anis(x'+R p_1)^{N-1}},
\ g(x') = \frac{1}{\Anis(x'+R p_0)^{N-1}}.
\end{align*}
Then, using the convexity of $\Anis$, we have for all $x',y'$,
\[
h(\lambda x'+(1-\lambda y'))\ge M_{-1/(N-1)}(f(x'),g(y'),\lambda)
\]
where $M_p$ is defined in~\cite[p.~368]{gardner} by
\begin{equation}\label{eq:defMp}
M_p(a,b,\lambda) = (\lambda a^p +(1-\lambda)b^p)^{1/p}.
\end{equation}
Thanks to Borell-Braskamp-Lieb's inequality \cite[Thm.~10.1, (38)]{gardner} we have
\[
\int_{\R^{N-1}}h(x')dx'\ge M_q\left(\int_{\R^{N-1}}f(x')dx',\int_{\R^{N-1}}g(x')dx',\lambda\right)
\]
for
\[
q = \frac{-1/(N-1)}{-\frac{N-2}{N-1} + 1} = -1.
\]
But using~\eqref{eq:Phi}, this precisely boils down to
\[
\Phi(\lambda p_1+(1-\lambda)p_0)\le\lambda\Phi(p_1)+(1-\lambda)\Phi(p_0).
\]
\end{proof}

\section{Evolution of convex sets}\label{secconvex}
In this section we show that the during the flow the convexity of a set is preserved. The main result in this direction is contained in Lemma \ref{convexity}, where it is shown that in each step of the discrete approximation, the convexity is preserved. Such a Lemma is actually a consequence of a series of non-trivial results in convex geometry. We begin by recalling such results. All the above definitions and results can be found, together with a comprehensive list of  references, in the survey \cite{gardner}. 
\begin{definition}
	We say that a function $f\in L^1(\R^N)\cap C^0(\R^N)$ is a $p-$concave function if it is $\log-$concave when $p=0$ and, if $p\ne 0$, for every $x,y\in\R^N$ it holds
	\[
	f((1-\lambda)x+\lambda y)\ge M_p(\lambda,f(y),f(x))
	\]
where $M_p$ is defined in~\eqref{eq:defMp}.
	Equivalently, $f$ is $p-$concave if $f^p$ is convex, for $p$ negative, and $f^p$ is concave, for $p$ positive.
\end{definition}
For our scopes we will need the following result which is a consequence of \cite[Theorem $3.16$]{dhajoa}, \cite[Corollary 11.2]{gardner} (see the discussion at page $379$ of \cite{gardner}).
\begin{lemma}
	Let $f$ be a $p$-concave function with $p>-1/N$ and $K$ a convex body (that is a convex set with non-empty interior). Then the function 
	\[
	h(x)=f*\chi_K(x)
	\]
	is $(1/Np+1)$-concave, and as a consequence, $h$ is level set convex.
\end{lemma}

\begin{corollary}\label{convexity}
If $E$ is convex then for any $h>0$, $c\in\R$, 
the set $T_{c,h}(E)$ (defined in \eqref{schema}) is convex.
\end{corollary}

\begin{proof}
We only have to notice that the function $P_h$ is continuous, integrable over $\R^N$, and that it is $-1/(N+s)$-concave. The latter property follows by a direct inspection.
\end{proof}
\begin{corollary}\label{convexity1}
 Let $g:\R_+\to\R$ be a continuous function. Let $u_0$ be a regular function such that all level sets $\{u_0>s\}$ are convex. If $u$ is the solution of \eqref{nlMCF} with initial data $u_0$, then the level sets $\{u(\cdot,t)>s \}$ 
are convex.
\end{corollary}
\begin{proof}
This follows from the fact that, thanks to Theorem~\ref{mainthm} and
Remark~\ref{rem:geom}, (almost) all the level sets of $u(t)$ can be obtained
as limits of the scheme~\ref{schema}, which preserves convexity.
\end{proof}

\section{A splitting result}\label{secsplit}

The goal of this section is to show that the motion with forcing term
can be obtained by alternating free curvature motions and evolutions
with the forcing term only. A consequence will be an elementary proof
of how the distance between two sets evolve by the forced curvature
flow (as this distance increases by unforced mean curvature flow, and
its evolution is trivial for sets evolving with constant speeds), see
Prop.~\ref{propdist} below.

Let $g:\R_+\to\R$ be a continuous function.
For a fixed $\e>0$ consider the sets $A_\e=\cup_{n\ge 0}(2n\e,(2n+1)\e]$ and $B_\e=(0,\infty)\setminus A_\e$. Let, for $t>0$, $p\in\R^N$ and $I\in\R$,
\begin{equation}\label{Fvarepsilon}
F_\e(t,p,I):=2\chi_{A_\e}(t)\Phi(p)c_\e(t)+2\chi_{B_\e}(t)\Phi(p)I, 
\end{equation}
where $c_\e$ is the piecewise constant function defined by 
\[
c_\e(t)=\frac{1}{2\e}\int_{2n\e}^{(2n+2)\e}g(\tau)\,d\tau,
\]
if $t\in[2n\e,(2n+2)\e]$, and where $\Phi (p)=\As(p)|p|$ is
the mobility (see Section \ref{convexityofthemobility}).
We let also
\begin{equation}\label{F}
F(t,p,I):=\Phi(p)(I+g(t)), 
\end{equation}
and we observe that the function $t\mapsto \int_0^t( F_\e(\tau,p,I)-F(\tau,p,I))d\tau$
goes locally uniformly to $0$ as $\e\to 0$ (for fixed $p,I$).

Let $u_0:\R^N\to\R$ a bounded uniformly continuous function and $u_\e:\R^N\times[0,\infty)$ be the function constructed as follows. 
We let $u_\e(\cdot,0)=u_0$ and for each $n$, define $u(\cdot,t )$ on
$(n\e,(n+1)\e]$ as the (unique) viscosity solution, starting
from $u_\e(n\e)$, of
\begin{equation}\label{eq:eqsplit}
\partial_t u_\e=F_\e(t,Du_{\e},-\kappa_s(x,\{u_\e\ge u_\e(x,t)\})).
\end{equation}

It is easy to see that $u_\e$ remains bounded and spatially
uniformly continuous,
moreover it is classical that it is also uniformly continuous in time
(see for instance~\cite{imbert}).
Hence up to a subsequence, we may assume that it converges uniformly,
as $\e\to 0$, to a continuous limit $u(x,t)$. We will show that
$u$ is the solution of~\eqref{nlMCF}. Since this limit is independent
of the chosen subsequence, it will yield the following lemma.

\begin{proposition}\label{splittingthm}
Let $u_\e,\,u$ be respectively the solutions of \eqref{eq:eqsplit}, \eqref{nlMCF},
with initial datum $u_0$. Then $u_\e\to u$,
as $\e\to 0$, locally uniformly in $\R^N\times [0,\infty)$.
\end{proposition}

\begin{proof}
We just need to show that $u$, limit of a subsequence
of $(u_\e)$, satisfies~\eqref{eq:subsol0}.
The proof is based on a convergence result of Barles~\cite{bar},
based the theory of $L^1$-viscosity solution \cite{ishii,bourgoing}.
We adapt it to our nonlocal setting, and simplify significantly the argument,
as we do not wish to show a very general convergence result for
nonlocal geometric motions (even if this could be of independent interest).

Consider $\ph$ and $(\bar x,\bar t)$ a strict global maximum of
$u-\ph$, and assume first $D\ph(\bar x,\bar t)\neq 0$
and $\ph(\bar x,\bar t)$ is not a critical value
of $\ph$. As in~\cite{bar}, we introduce
\[
\psi_\e(t) := 
F_\e(t,D\ph(\bar x,\bar t),-\kappa_s(\bar x,\{\ph\ge \ph(\bar x,\bar t)\})) 
- F(t,D\ph(\bar x,\bar t),-\kappa_s(\bar x,\{\ph\ge \ph(\bar x,\bar t)\}))
\]
which is such that $\int_0^t\psi_\e(\tau)d\tau\to 0$ uniformly. Hence,
$u_\e(x,t)-\int_0^t\psi_\e(\tau)d\tau\to u(x,t)$ locally uniformly, and
one can find points $(x_\e,t_\e)$ of global maximum of 
\[
u_\e(x,t)-\int_0^t\psi_\e(\tau)d\tau - \ph(x,t),
\]
such that $(x_\e,t_\e)\to(\bar x,\bar t)$ as $\e\to 0$.
If $t/\e\not\in\N$, one deduces that
\begin{equation}\label{eq:visceps}
\partial_t\ph(x_\e,t_\e)+\psi_\e(t_\e)\le F_\e(t_\e,D\ph(x_\e,t_\e),
-\kappa_s(x_\e,\{\ph\ge \ph(x_\e,t_\e)\})),
\end{equation}
observing in particular that since $(x_\e,t_\e)\to(\bar x,\bar t)$,
the value $\ph(x_\e,t_\e)$ can be assumed to be noncritical\footnote{Strictly
speaking, there could be critical points of the corresponding level sets,
however these points tend to infinity as $\e\to 0$,
and do not alter significantly
the value of the integrals defining $\kappa_s$.}
If $t_\e/\e\in\N$, classical arguments for parabolic semigroups
show that~\eqref{eq:visceps} still holds, if one takes
for $\psi_\e$ and $F_\e$ their left limit, see for instance \cite{ishiisouganidis}. It follows
\begin{multline}\label{eq:spliterror}
\partial_t\ph(x_\e,t_\e)
\le F(t_\e,D\ph(\bar x,\bar t),-\kappa_s(\bar x,\{\ph\ge \ph(\bar x,\bar t)\}))
\ +
\\   \big[F_\e(t_\e,D\ph(x_\e,t_\e),-\kappa_s(x_\e,\{\ph\ge \ph(x_\e,t_\e)\}))
-F_\e(t_\e,D\ph(\bar x,\bar t),-\kappa_s(\bar x,\{\ph\ge \ph(\bar x,\bar t)\})) \big].
\end{multline}
As
\[
\lim_{\e\to 0} D\ph(x_\e,t_\e) = D\ph(\bar x,\bar t)
\quad\text{and}\quad
\lim_{\e\to 0}\kappa_s(x_\e,\{\ph\ge \ph(x_\e,t_\e)\})=\kappa_s(\bar x,\{\ph\ge \ph(\bar x,\bar t)\}),
\]
the error term in square brackets in~\eqref{eq:spliterror} vanishes in the limit and it follows
\[
\partial_t\ph(\bar x,\bar t)
\le F(t,D\ph(\bar x,\bar t),-\kappa_s(\bar x,\{\ph\ge \ph(\bar x,\bar t)\})),
\]
which is~\eqref{eq:subsol0}.

If on the other hand $D\ph(\bar x,\bar t)=0$, then the proof
that $\partial_t\ph(\bar x,\bar t)\le 0$ is identical to Step~$2$
in the proof of Proposition~\ref{consistency},
provided one can first estimate the speed at which balls evolve
under the equation~\eqref{eq:eqsplit}, which is of the same order as in Corollary~\ref{corboundedspeed}.
\end{proof}

\section{Geometric uniqueness in the convex case}\label{secunique}
In this section we show that if the initial set is bounded and convex, then the fattening phenomenon can not occur and the evolution is unique. 
The proof is based on~\cite[Theorem $4.9$]{belcaschanov2},
and follows from a (simple) estimate of the distance between two evolutions
with different forcing terms.
In the rest of the paper, we will always consider in $\R^N$ the distance
$\dist_{\Phio}$
induced by the norm $\Phio(x):=\sup\{\xi\cdot x:\Phi(\xi)\le 1\}$,
polar of $\Phi$. Hence we will drop the subscript and write $\dist$ instead
of $\dist_{\Phio}$. Similarly, we will write $d^\eta_{C}=-\eta\vee (\eta\wedge (\dist_{\Phio}(x,C)-\dist_\Phio(x,C^c)))$.

\begin{lemma}\label{d>c-c}
Let $C_1\subseteq C_2$ two sets and let $C_1(t)$ and $C_2(t)$ be the evolutions of the flow $v_i=\Phi(\nu)c_i$, with $c_i$ two constants, starting from $C_1$ and $C_2$ respectively. That is, $C_i(t)=\{u_i\ge t\}$ where $u_i$ is the solution of
\[
\begin{cases}
\partial_t u_i=c_i\Phi(Du_i)\\
u_i(x,0)=d^\eta_{C_i}(x).
\end{cases}
\]
Then 
the function
\[
\delta(t):=\dist_{\Phio}(\partial C_1(t),\partial C_2(t))
\] 
satisfies
\[
\delta(t)\ge\delta(0)-t  (c_2-c_1),
\]
for every $0\le t \le T_S:=\sup\{\tau\ge 0:\delta(\tau)>0\}$ (i.e., until the first contact time).
\end{lemma}

\begin{proof}  
We consider first the case where $c_1$ and $c_2$ are not positive. We recall that, by the Hopf-Lax formula for the Hamiltonian $H_i(p)=|c_i|\Phi(p)$, 
the solution of the system
\[
\begin{cases}
 \partial_t u_i(x,t)+|c_i|\Phi(Du(x,t))=0\\
 u_i(x,0)=d^\eta_{C_i}(x),
\end{cases}
\]
with $i=1,2$, is given by  (see for instance \cite{evans})
\[
u_i(x,t)=\inf_{y\in\R^N}\left\{d^\eta_{C_i}(x)+ tH_i^*\left(\frac{x-y}{t}\right)\right\},
\]
where $H_i^*$ denotes the Legendre-Fenchel transform of the function $H_i$, 
given by
\[
H_i^*(q) = 
\begin{cases} 0 & \textup{ if }\Phio(q)\le |c_i|,\\
+\infty & \textup{ else.}
\end{cases}
\]
Thus
\[
u_i(x,t)=\inf_{y:\Phio(y-x)\le |c_i|t}d^\eta_{C_i}(x).
\]
Since $\delta(0)>0$ we can suppose  that $t$ is such that $C_1(t)\subset C_2(0)$.
We have that 
\begin{equation}\label{pranzo}
\{x: u_i(x,t)>0\}=\left\{C_i+B_\Phio(0,|c_i| t)\right\}^c.
\end{equation} 
\noindent Indeed if $\xi\in \{x: u_i(x,t)>0\}$ then $\xi$ can not belong to $C_i(t)+B_\Phio(0,|c_i| t)$, otherwise there would exist $z\in \partial C_i$ with $\Phio(z-\xi)\le |c_i|t$ and thus 
\[
0=\dist(z,\partial C_i)\ge u_i(\xi,t)>0.
\]
On the other hand it is immediate to verify that if $\dist(\xi,C_i)\le |c_i|t$ then $u(\xi,t)\le 0$.

Let $x_i\in \partial (C_i+B_\Phio(0,|c_i| t))$ be such that $\delta(t)=\Phio(x_1-x_2)$ and denote $\xi$ the unique intersection  between $\partial C_2$ and the segment with extrema $x_1$ and $x_2$.
Let moreover $z$ be the projection of $x_1$ onto $\partial C_1$, so that
$\Phio(x_1-z)=|c_1|t$. We have
\[
\begin{aligned}
\delta(t)&=\Phio(x_1-x_2)\\
& =\Phio(x_1-\xi)+\Phio(\xi-x_2)\\
&\ge |c_2|t+\Phio(\xi-x_1)\\
& \ge |c_2|t+\Phio(\xi-z)-\Phio(z-x_1)\\
& =|c_2|t+\delta(0)-|c_1|t\\
&=\delta(0)-(c_2-c_1)t,
\end{aligned}
\]  
which is exactly the statement of the lemma. The proof in the case where $c_1$ and $c_2$ are positive follows the same lines of the above proof, once we notice that if $u$ solves $\partial_t u -|c|\Phi(Du)=0$ then $v=-u$ solves $\partial_t v+|c|\Phi(Dv)=0$. If $c_1<0$ and $c_2>0$ by similar arguments  we get that
\[
C_1(t) =\left\{C_1+B_\Phio(0,|c_1|t)\right\}^c \qquad
C_2(t)=\{x:\dist(x,\partial C_2>c_2t)\}.
\]
Let $x_1$ and $x_2$ be points such that $\delta(t)=\Phio(x_1-x_2)$ and let $\overline{x_i}$ be the projection of $x_i$ to $\partial C_i$, $i=1,2$. Then we have
\[
\begin{aligned}
\delta(t)&=\Phio(x_1-x_2)\\
&\ge \Phio(\overline{x_1}-\overline{x_2})-\Phio(x_1-\overline{x_1})
-\Phio(x_2-\overline{x_2})\\
&\ge \delta(0)-|c_1|t-|c_2|t\\
&=\delta(0)-(c_2-c_1)t.
\end{aligned}
\]
The proof of the case $c_1>0$ and $c_2<0$ follows by an analogous argument.
\end{proof}

\begin{proposition}\label{propdist}
	Let $C_1\subseteq C_2$ be two sets and let, $g_1$ and $g_2:\R_+\to\R$ two continuous functions and, for $i=1,2$, let $u_i$ be the solution of 
	\[
	\begin{cases}
	\partial_t u_i=\Phi(Du_i)  \left( -\kappa_s(x,\{u_i\ge u_i(x,t)\})+g_i(t)\right),\\
	u_i(x,0)=d^\eta_{C_i}(x)
	\end{cases}.
	\]
 Let for $t\ge 0$, $i=1,2$, $C_i(t)=\{u_i\ge 0\}$.
Then the function $\delta(t)=\dist(\partial C_1,\partial C_2)$ satisfies
	\[
	\delta(t)\ge \delta(0)-  \int_0^t \left(c_2(s)-c_1(s)\right)\,ds
	\]
	for every $0\le t \le T_S:=\sup\{\tau\ge 0:\delta(\tau)>0\}$.
\end{proposition}

\begin{proof}
Without loss of generality we can assume that the $0$-level sets of $u_i$ do not {\it fatten},
that is, $C_i(t)=\overline{\{u_i> 0\}}$ (otherwise we should consider the $\tau$-level set and then let $\tau \to 0$).

For $i=1,2$, let $u_{\e,i}$ be the functions constructed in Section \ref{secsplit} with $g=g_i$,
let $c_{\e,i}$ be the corresponding piecewise forcing terms,
and let $C^i_\e(t)=\{x:u_{\e,i}(x,t)\ge 0\}$ and $\delta_\e(t)=\dist(\partial C_\e^1(t),C_\e^2(t))$. By Lemma~\ref{splittingthm} we have that $\delta_\e(t)\to\delta(t)$ for every $t\le T_S$.

Let $t\le T_S$ and let $n$ be the largest integer such that $n\e<t$. Let us write $\delta_\e(t)$ as
	\[
	\delta_\e(t)=\delta_\e(\e)+\left[\delta_\e(2\e)-\delta_\e(\e)\right]+\left[\delta_\e(3\e)-\delta_\e(2\e)\right]+\dots+\left[\delta_\e(t)-\delta_\e(n\e)\right]. 
	\]
	Since the functions $u_{\e, i}$ solve in $[0,\e]$ the  geometric and
        translation-invariant equation
	$\partial_t u_{\e,i}=2\Phi(Du_{\e,i}) (-\kappa_s(x,\{u_{\e,i}\ge u_i(x,t)\}))$,
        the distance between their $0$-level sets is nondecreasing, so that
	$\delta_\e(\e)\ge\delta_\e(0)$.
	Moreover, since the $u_{\e,i}$'s solve in $(\e,2\e]$ the equation 
	$\partial_t u_{\e,i}=2\Phi(Du_{\e,i}) c_{\e,i}$, by Lemma \ref{d>c-c}
	we get that 
\[
\delta_\e(2\e)\ge\delta_\e(\e)-2\e\left(c_{\e,2}(2\e)-c_{\e,1}(2\e)\right).
\]
By iterating this argument we obtain that
\[
\begin{cases}
\delta_\e(k\e)-\delta_\e((k-1)\e)\ge0&\text{if $k$ is odd,}\\
\delta_\e(k\e)-\delta_\e((k-1)\e)\ge-2  \e\left(c_{\e,2}(k\e)-c_{\e,1}(k\e)\right)&\text{otherwise}.
\end{cases}
\]
By summing in $k$ we then get
\[
\delta_\e(n\e)\ge\delta_\e(0)+2  \e\sum_{k=1}^{n}\left(c_{\e,2}(k\e)-c_{\e,1}(k\e)\right)
=\delta_\e(0)+\int_0^{n\e}\left(g_{2}(\tau)-g_{1}(\tau)\right)\,d\tau\,.
\]
By passing to the limit as $\e\to 0$, the thesis follows.
\end{proof}


Thanks to the previous proposition, by reasoning exactly as in the proofs of \cite[Theorem~4.9]{belcaschanov2} (which is based in turn on \cite[Theorem 8.4]{belcaschanov1}),
 we get the following corollary:

\begin{corollary}
	Let $g:\R\to\R$ be a continuous function. Let $C_1\subseteq C_2$ be two compact convex sets and let $C_1(t)$ and $C_2(t)$ be the flows for the equation $\As^{-1}v=-\kappa_s+g$, starting from $C_1$ and $C_2$ respectively. Then $C_1(t)\subseteq C_2(t)$ for all $t\ge 0$.
\end{corollary}
\begin{proof}
Notice first that, if $C_1$ has empty interior, then $C_1(t)=\emptyset$ for all $t>0$,
and there is nothing to prove.
Therefore, we can assume that $C_1$ has nonempty interior 
and that $0$ lies in the interior of ${C}_1$.

For $\theta>1$, we let $g_1(t):=g(t)$, $g_2(t):=g(t/\theta^{1+s})/\theta^s$.
Notice that the set $\theta C_2(t/\theta^{1+s})$ solves the equation
$\As^{-1}v=-\kappa_s+g(t/\theta^{1+s})/\theta^s$, with initial datum $\theta C_2$.
Therefore,
letting $\delta_\theta(t):=\dist (\partial C_1(t),\theta\partial C_2(r/\theta^{1+s}))$, 
so that $\delta_\theta(0)>0$,
by Proposition~\ref{propdist} we get
\[
\delta_\theta(t) \ge \delta_\theta(0)-\int_0^t \left(
\frac{1}{\theta^s}g\left(\frac{\tau}{\theta^{1+s}}\right)-g(\tau)\right)d\tau
\]
until the first contact time. Now, the integral can be estimated by
\[
(\theta-1) \int_0^{t/\theta^{1+s}} g(\tau)d\tau+\int_{t/\theta^{1+s}}^t
g(\tau)d\tau\le 3t(\theta-1)\|g\|_\infty,
\]
while $\delta_\theta (0) \ge c(\theta-1)$ where $c$ depends only on the
$C_1$ and $\Phi$. Hence, $\delta_\theta(t)\ge 0$
as long as $t\le c/(3\|g\|_\infty)$, which does not depend on $\theta$.
It follows that $C_1(t)\subseteq C_2(t)$, this as long as $C_1$ has
nonempty interior, which concludes the proof.
\end{proof}
\begin{remark}
The same proof shows that, in general, a strictly star-shaped domain with
respect to a center point $x$ will
have a unique evolution for a positive time,
as long as no line issued from $x$ becomes tangent to its boundary.
\end{remark}

\section{Concluding remarks}\label{secremarks}
A natural question is whether one can characterize the sets which
evolve homothetically by the anisotropic flow \eqref{nlMCF}, with $g=0$. A way to build
such sets could be by first showing existence of evolutions with constant
volume
(by tuning appropriately the forcing term as in~\cite{belcaschanov2,lauxswartz})
and then studying their asymptotic limit. Anyway, the characterization of the limiting 
equilibrium shape seems
to be a difficult question, related to the anisotropic fractional
isoperimetric problem. 

More precisely, it is known (see~\cite{ludwig}) that 
the  $\mathcal N-$fractional perimeter converges, as $s\to 1$, to an anisotropic perimeter 
with a specific anisotropy different from $\mathcal N$,
yielding an indication on the behavior of the isoperimetric sets in this limit.

Another natural question is whether the $\mathcal N-$fractional perimeter 
is decreasing under the limit flow, in absence of the forcing term.
This is true in the isotropic case, and can be easily seen by writing the flow in~\cite{cafsou}
as a minimizing movement scheme as
in~\cite{lauxswartz,esedogluotto}, and is probably true also in our
case. However a complete proof would require a
thorough study of the properties of the kernel $P_h$.

In this respect, it would be interesting to extend the analysis in~\cite{lauxswartz,esedogluotto}
to the fractional case.


\section*{Acknowledgements}
A.C.~is partially supported by the ANR networks ``HJNet'' ANR-12-BS01-0008-01,
and ``Geometrya'' ANR-12-BS01-0014-01.
Most of this work was done while
M.N.~and B.R.~were hosted at the CMAP (Ecole Polytechnique and CNRS),
B.R.~was supported by a fellowship of the Fondation Math\'ematique Jacques
Hadamard and the LMH (ANR-11-LABX-0056-LMH), and M.N.~by a one-month
invitation of the Ecole Polytechnique. 
M.N. and B.R. were also partially supported by the University of Pisa via grant PRA-2015-0017.
A.C.~and~M.N.~also acknowledge
the hospitality of the MFO (Oberwolfach) where this work was completed.

\end{document}